\documentclass[twoside,a4paper,backrefs,msc-links]{amsproc}
\usepackage{graphicx}
\PassOptionsToPackage{pdfauthor={Vladimir V. Kisil},%
    pdftitle={Erlangen Program at Large: Outline},%
    pdfsubject={general mathematics},%
    backref=page,%
    breaklinks=true,
    pdfkeywords={symmetries, Erlangen program}
}{hyperref}
\usepackage{hyperref}
\usepackage{amsrefs}
\IfFileExists{eulervm.sty}{\usepackage{eulervm}
}{}
\numberwithin{equation}{section}
\newtheorem{thm}{Theorem}[section]
\newtheorem{prop}[thm]{Proposition}
\theoremstyle{definition}
\newtheorem{defn}[thm]{Definition}
\newtheorem{example}[thm]{Example}
\theoremstyle{remark}
\newtheorem{rem}[thm]{Remark}

\makeatletter
\providecommand{\diag}{\mathop {\operator@font diag}\nolimits}
\makeatother
\providecommand{\object}[2][\,]{\ensuremath{\mathrm{#2}#1}}
\providecommand{\such}{\,\mid\,}
\providecommand{\norm}[2][\relax]{\left\|#2\right\|\ifx#1\relax\else_{#1}\fi}
\providecommand{\modulus}[2][\relax]{\left| #2 \right|\ifx#1\relax\else_{#1}\fi}
\providecommand{\lvec}[1]{\overrightarrow{#1}}

\providecommand{\cycle}[3][]{{#1 C^{#2}_{#3}}}
\newcommand{\zcycle}[3][]{#1 Z^{#2}_{#3}}
\newcommand{\realline}[3][]{#1 R^{#2}_{#3}}
\providecommand{\matr}[4]{{\ensuremath{ \left(\!\! \begin{array}{cc}
#1 & #2 \\ #3 & #4
\end{array}\!\!\right) }}}
\providecommand{\spec}[1][]{\ensuremath{\mathbf{sp}}\,}
\providecommand{\GiNaC}{\textsf{GiNaC}}
\providecommand{\bs}{\breve{\sigma}}
\providecommand{\SL}[1][2]{\ensuremath{\FSpace{SL}{#1}(\Space{R}{})}}
\providecommand{\scalar}[3][\relax]{\left\langle #2,#3 
        \right\rangle\ifx#1\relax\else_{#1}\fi}
\providecommand{\Space}[3][]{\ensuremath{\mathbb{#2}^{#3}_{#1}{}}}
  \providecommand{\FSpace}[3][]{\ensuremath{\ifx#2l \ell_{#3}^{#1}{}\else
  #2_{#3}^{#1}{}\fi}} 
\providecommand{\rmi}{\mathrm{i}}
\providecommand{\rmc}{\mathrm{\breve\i}}
\providecommand{\tr}{\mathop{tr}}
\providecommand{\algebra}[1]{\ensuremath{\mathfrak{#1}}}
\providecommand{\MR}[1]{\textbf{MR}~\href{http://www.ams.org/mathscinet-getitem?mr=#1}{\#~#1}}
\providecommand{\Zbl}[1]{\textbf{Zbl}~\href{http://www.emis.de:80/cgi-bin/zmen/ZMATH/en/zmathf.html?first=1&maxdocs=3&type=html&an=#1&format=complete}{\#~#1}}
\providecommand{\eprint}[2]{E-print: \href{#1}{\texttt{#2}}}
\providecommand{\modulus}[2][\relax]{\left| #2 \right|\ifx#1\relax\else_{#1}\fi}
\hypersetup{colorlinks=true,bookmarks=true}
\providecommand{\wiki}[2]{\href{http://en.wikipedia.org/wiki/#1}{#2}}
\newcommand{\Ba}{\bar{\alpha}}
\newcommand{\Bb}{\bar{\beta}}

\providecommand{\oper}[1]{\mathcal{#1}}
\newcommand{\ga}{\mathsf{a}}

\begin{document}

\title
{Erlangen Program at Large: Outline}

\author[Vladimir V. Kisil]%
{\href{http://maths.leeds.ac.uk/~kisilv/}{Vladimir V. Kisil}}
\address{
School of Mathematics,
University of Leeds,
Leeds, LS2\,9JT,
UK
}
\thanks{On  leave from the Odessa University.}
\email{
\href{mailto:kisilv@maths.leeds.ac.uk}{kisilv@maths.leeds.ac.uk}
}
\urladdr{
\url{http://www.maths.leeds.ac.uk/~kisilv/}}

\subjclass[2000]{Primary 30G35; Secondary 22E46, 30F45, 32F45, 43A85, 30G30, 42C40, 46H30, 47A13, 81R30, 81R60.}
\keywords{  Special linear group, Hardy space, Clifford algebra, elliptic,
  parabolic, hyperbolic, complex numbers, dual
  numbers, double numbers, split-complex numbers,
  Cauchy-Riemann-Dirac operator, M\"obius transformations, functional
  calculus, spectrum, quantum mechanics, non-commutative geometry.}

\maketitle

\begin{abstract}
  This is an outline of \emph{Erlangen Program at
    Large}. Study of objects and properties, which are invariant under
  a group action, is very fruitful far beyond the traditional
  geometry. In this paper we demonstrate this on the example of the
  group \(\SL\). Starting from the conformal geometry we develop
  analytic functions and apply these to functional calculus. Finally we
  provide an extensive description of open problems.
\end{abstract}
\tableofcontents

\section{Introduction}
\label{sec:introduction}

  The simplest objects with non-commutative multiplication may be
  \(2\times 2\) matrices with real entries.  Such matrices \emph{of
    determinant one} form a closed set under multiplication (since
  \(\det (AB)=\det A\cdot \det B\)), the identity matrix is among them
  and any such matrix has an inverse (since \(\det A\neq 0\)). In
  other words those matrices form a group, \wiki{SL2(R)}{the \(\SL\)
    group}~\cite{Lang85}---one of the two most important Lie groups
  in analysis. The other group is \wiki{Heisenberg_group}{the
    Heisenberg group}~\cite{Howe80a}. By contrast the
  \wiki{Affine_transformation}{``\(ax+b\)''-group}, which is often
  used to build wavelets, is only a subgroup of \(\SL\), see the
  numerator in~\eqref{eq:moebius}.

  The simplest non-linear transforms of the real
  line---linear-fractional or \wiki{Moebius_transformation}{M\"obius
    maps}---may also be associated with \(2\times 2\)
  matrices~\cite{Beardon05a}*{Ch.~13}:
  \begin{equation}
    \label{eq:moebius}
    g: x\mapsto g\cdot x=\frac{ax+b}{cx+d}, \text{ where } 
    g=  \begin{pmatrix}
      a&b\\c&d
    \end{pmatrix}, x\in\Space{R}{}.
  \end{equation}
  An enjoyable calculation shows that the composition of two
  transforms~\eqref{eq:moebius} with different matrices \(g_1\) and
  \(g_2\) is again a M\"obius transform with matrix the product
  \(g_1 g_2\). In other words~\eqref{eq:moebius} it is a (left) action
  of \(\SL\).

  According to F.~Klein's \wiki{Erlangen_program}{\emph{Erlangen
      program}} (which was influenced by S.~Lie) any geometry is
  dealing with invariant properties under a certain group action. For
  example, we may ask: \emph{What kinds of geometry are related to
   the \(\SL\) action~\eqref{eq:moebius}}?
  
  The Erlangen program has probably the highest rate of
  \(\frac{\text{praised}}{\text{actually used}}\) among mathematical
  theories not only due to the big numerator but also due to undeserving
  small denominator. As we shall see below Klein's approach provides
  some surprising conclusions even for such over-studied objects as
  circles.

  \subsection{Make a Guess in Three Attempts}
  \label{sec:make-guess-three}

  It is easy to see that the \(\SL\) action~\eqref{eq:moebius} makes
  sense also as a map of complex numbers \(z=x+\rmi y\),
  \(\rmi^2=-1\). Moreover, if \(y>0\) then \(g\cdot z\) has a positive
  imaginary part as well, i.e. \eqref{eq:moebius} defines a map from
  the upper half-plane to itself. 

  However there is no need to be restricted to the traditional route
  of complex numbers only. Less-known \wiki{Dual_number}{\emph{dual}}
  and \wiki{Split-complex_number}{\emph{double}} numbers
  \cite{Yaglom79}*{Suppl.~C} have also the form \(z=x+\rmi y\) but
  different assumptions on the imaginary unit \(\rmi\): \(\rmi^2=0\)
  or \(\rmi^2=1\) correspondingly. Although the arithmetic of dual and
  double numbers is different from the complex ones, e.g. they have
  divisors of zero, we are still able to define their transforms
  by~\eqref{eq:moebius} in most cases.

  Three possible values \(-1\), \(0\) and \(1\) of \(\sigma:=\rmi^2\)
  will be refereed to here as \emph{elliptic}, \emph{parabolic} and
  \emph{hyperbolic} cases respectively.  We repeatedly meet such a
  division of various mathematical objects into three classes.  They
  are named by the historically first example---the classification of
  conic sections---however the pattern persistently reproduces itself
  in many different areas: equations, quadratic forms, metrics,
  manifolds, operators, etc.  We will abbreviate this separation as
  \emph{EPH-classification}.  The \emph{common origin} of this
  fundamental division can be seen from the simple picture of a
  coordinate line split by zero into negative and positive
  half-axes:
  \begin{equation}
    \label{eq:eph-class}
    \raisebox{-15pt}{\includegraphics[scale=1]{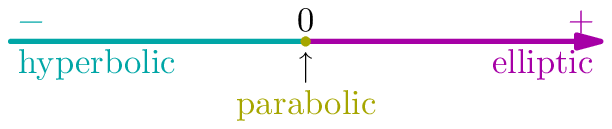}}
  \end{equation}

  Connections between different objects admitting EPH-classification
  are not limited to this common source. There are many deep results
  linking, for example,
  \wiki{Atiyah-Singer_index_theorem}{ellipticity
    of quadratic forms, metrics and operators}.
 On the other hand there are still a lot of white spots and obscure
 gaps between some subjects as well.

  To understand the action~\eqref{eq:moebius} in all EPH cases we use
  the Iwasawa decomposition~\cite{Lang85} of \(\SL=ANK\) into
  \emph{three} one-dimensional subgroups \(A\), \(N\) and 
  \(K\):
  \begin{equation}
    \label{eq:iwasawa-decomp}
    \begin{pmatrix}
      a&b \\c &d
    \end{pmatrix}= {\begin{pmatrix} \alpha & 0\\0&\alpha^{-1}
      \end{pmatrix}} {\begin{pmatrix} 1&\nu \\0&1
      \end{pmatrix}} {\begin{pmatrix}
        \cos\phi &  \sin\phi\\
        -\sin\phi & \cos\phi
      \end{pmatrix}}.
  \end{equation}
  Subgroups \(A\) and \(N\) act in~\eqref{eq:moebius} irrespectively
  to value of \(\sigma\): \(A\) makes a dilation by \(\alpha^2\), i.e.
  \(z\mapsto \alpha^2z\), and \(N\) shifts points to left by \(\nu\),
  i.e. \(z\mapsto z+\nu\).

\begin{figure}[htbp]
  \centering
  \includegraphics[scale=.5]{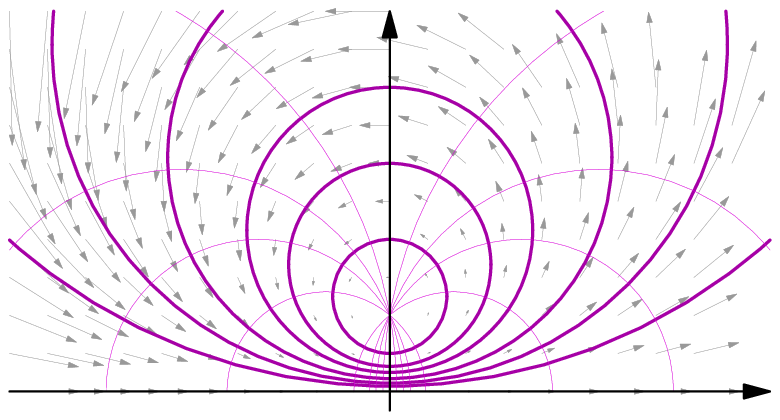}\hfill
  \includegraphics[scale=.5]{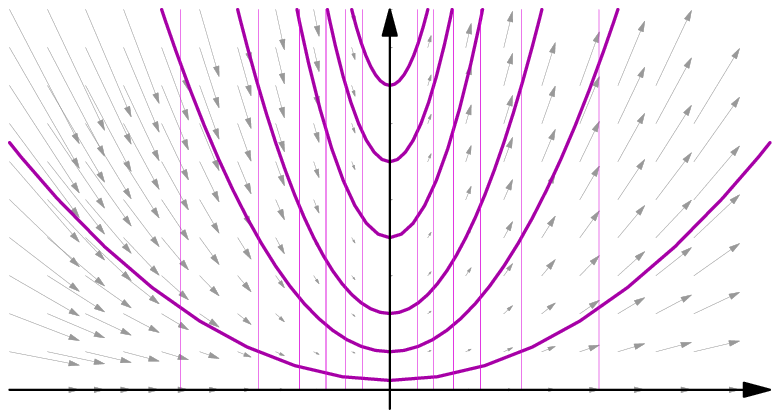}\hfill
  \includegraphics[scale=.5]{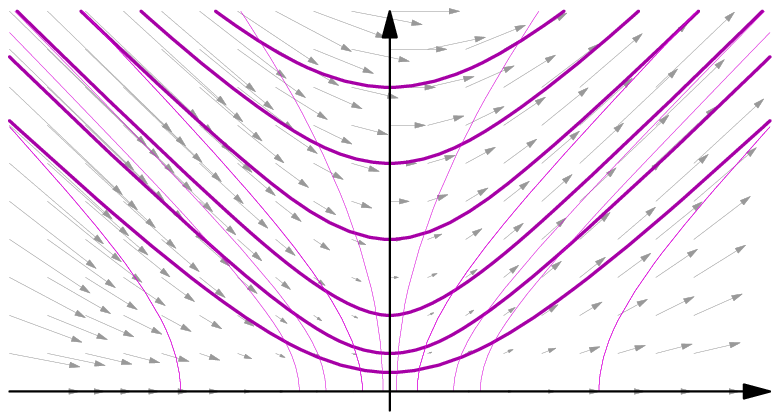}
  \caption[Action of the $K$ subgroup]{\small Action of the \(K\) subgroup.
    The corresponding \(K\)-orbits are thick circles, parabolas and
    hyperbolas. Thin traversal lines are images of the vertical axis
    for certain values of the parameter \(\phi\).}
  \label{fig:k-subgroup}
\end{figure}

  By contrast, the action of the third matrix from the subgroup \(K\)
  sharply depends on \(\sigma\), see Fig.~\ref{fig:k-subgroup}. In
  elliptic, parabolic and hyperbolic cases \(K\)-orbits are circles,
  parabolas and (equilateral) hyperbolas correspondingly.  Thin
  traversal lines in Fig.~\ref{fig:k-subgroup} join points of orbits
  for the same values of \(\phi\) and grey arrows represent ``local
  velocities''---vector fields of derived representations.

\subsection{Erlangen program at large}
  \label{sec:erlangen-program-at}

  As we already mentioned the division of mathematics into areas is
  only apparent. Therefore it is unnatural to limit Erlangen program
  only to ``geometry''. We may continue to look for \(\SL\) invariant
  objects in other related fields. For example,
  transform~\eqref{eq:moebius} generates unitary
  representations on certain \(\FSpace{L}{2}\) spaces, cf.~\eqref{eq:moebius}:
  
  \begin{equation}
    \label{eq:hardy-repres}
    g^{-1}: f(x)\mapsto \frac{1}{(cx+d)^m}f\left(\frac{ax+b}{cx+d}\right).
  \end{equation}

  For \(m=1\), \(2\), \ldots the invariant subspaces of
  \(\FSpace{L}{2}\) are Hardy and (weighted) Bergman spaces of complex
  analytic functions.  All main objects of \emph{complex analysis}
  (Cauchy and Bergman integrals, Cauchy-Riemann and Laplace equations,
  Taylor series etc.) may be obtaining in terms of invariants of the
  \emph{discrete series} representations of
  \(\SL\)~\cite{Kisil02c}*{\S~3}.  Moreover two other series
  (\emph{principal} and \emph{complimentary}~\cite{Lang85}) play the
  similar r\^oles for hyperbolic and parabolic
  cases~\citelist{\cite{Kisil02c} \cite{Kisil05a}}.

  Moving further we may observe that transform~\eqref{eq:moebius} is
  defined also for an element \(x\) in any algebra \(\algebra{A}\)
  with a unit \(\mathbf{1}\) as soon as
  \((cx+d\mathbf{1})\in\algebra{A}\) has an inverse. If
  \(\algebra{A}\) is equipped with a topology, e.g. is a Banach
  algebra, then we may study a \emph{functional calculus} for element
  \(x\)~\cite{Kisil02a} in this way. It is defined as an intertwining
  operator between the representation~\eqref{eq:hardy-repres} in a
  space of analytic functions and a similar representation in a left
  \(\algebra{A}\)-module.

  In the spirit of Erlangen program such functional calculus is still
  a geometry, since it is dealing with invariant properties under
  a group action. However even for a simplest non-normal operator, e.g.
  a Jordan block of the length \(k\), the obtained space is not like a
  space of point but is rather a space of \(k\)-th
  \emph{jets}~\cite{Kisil02a}. Such non-point behaviour is oftenly
  attributed to \emph{non-commutative geometry} and Erlangen program
  provides an important input on this fashionable topic~\cite{Kisil02c}.

  Of course, there is no reasons to limit Erlangen program to \(\SL\)
  group only, other groups may be more suitable in different
  situations.  However \(\SL\) still possesses a big unexplored potential
  and is a good object to start with.

\section{Geometry}
\label{sec:geometry}

\subsection{Cycles as Invariant Objects}
\label{sec:cycles-as-invariant}

  \begin{defn}
    \label{de:cycle}
    The common name \emph{cycle}~\cite{Yaglom79} is used to
    denote circles, parabolas and hyperbolas (as well as straight lines
    as their limits) in the respective EPH case.
  \end{defn}

\begin{figure}[htbp]
  \centering
  (a)\includegraphics[scale=.9]
  {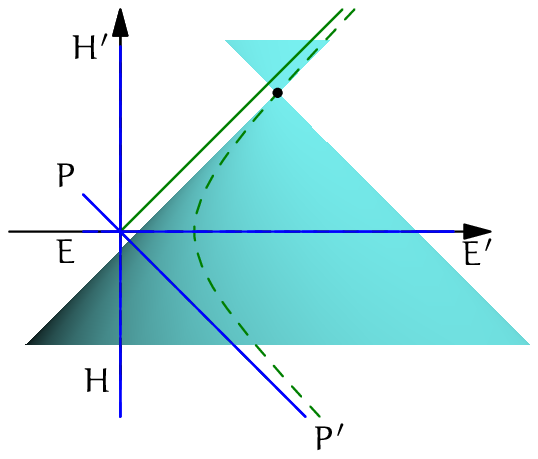}\hspace{2cm}
  (b)\includegraphics[scale=.9]
  {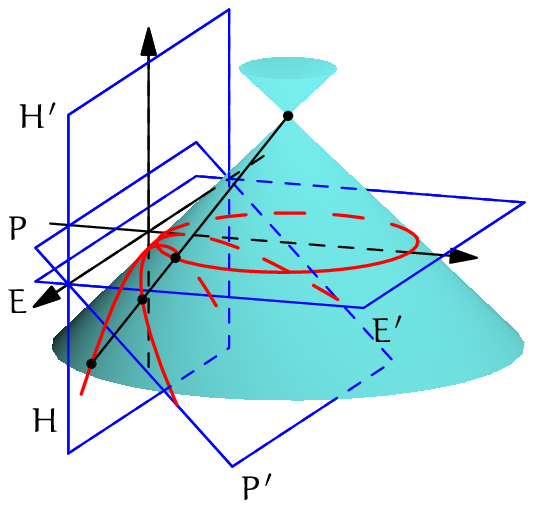}
  \caption[$K$-orbits as conic sections]{\small \(K\)-orbits as conic
    sections: 
    circles are sections by the plane \(EE'\);  parabolas are
    sections by \(PP'\);  hyperbolas are sections by \(HH'\). Points
    on the same generator of the cone correspond to the same value of \(\phi\).}
  \label{fig:k-orbit-sect}
\end{figure}

  It is well known that any cycle is a \emph{conic sections} and an
  interesting observation is that corresponding \(K\)-orbits are in
  fact sections of the same two-sided right-angle cone, see
  Fig.~\ref{fig:k-orbit-sect}.  Moreover, each straight line
  generating the cone, see Fig.~\ref{fig:k-orbit-sect}(b), is crossing
  corresponding EPH \(K\)-orbits at 
  points with the same value of parameter \(\phi\)
  from~\eqref{eq:iwasawa-decomp}. In other words, all three types of
  orbits are generated by the rotations of this generator along the
  cone.

  \(K\)-orbits are \(K\)-invariant in a trivial way. Moreover since
  actions of  both \(A\) and \(N\) for any \(\sigma\) are extremely
  ``shape-preserving'' we find natural invariant objects of the
  M\"obius map: 

  \begin{thm}[\cite{Kisil06a}]
    The family of all cycles from Defn.~\ref{de:cycle} is invariant under the
    action~\eqref{eq:moebius}.
  \end{thm}

  According to Erlangen ideology we shall 
  study invariant
  properties of cycles.

  \subsection{Invariance of FSCc}
  \label{sec:invariance-fscc}

  Fig.~\ref{fig:k-orbit-sect} suggests that we may get a unified
  treatment of cycles in all EPH by consideration of a higher
  dimension spaces. The standard mathematical method is to declare
  objects under investigations (cycles in our case, functions in
  functional analysis, etc.) to be simply points of some bigger
  space. This space should be equipped with an appropriate structure
  to hold externally information which were previously inner
  properties of our objects.
  
  A generic cycle is the set of points \((u,v)\in\Space{R}{2}\)
  defined for all values of \(\sigma\) by the equation
  \begin{equation}
    \label{eq:cycle-eq}
    k(u^2-\sigma v^2)-2lu-2nv+m=0.
  \end{equation}
  This equation (and the corresponding cycle) is defined by a point
  \((k, l, n, m)\) from a projective space \(\Space{P}{3}\), since for
  a scaling factor \(\lambda \neq 0\) the point \((\lambda k, \lambda
  l, \lambda n, \lambda m)\) defines the same
  equation~\eqref{eq:cycle-eq}. We call \(\Space{P}{3}\) the
  \emph{cycle space} and refer to the initial \(\Space{R}{2}\) as the
  \emph{point space}.

  In order to get a connection with M\"obius action~\eqref{eq:moebius}
  we arrange numbers \((k, l, n, m)\) into the matrix 
  \begin{equation}
    \label{eq:FSCc-matrix}
    C_{\bs}^s=\begin{pmatrix}
      l+\rmc s n&-m\\k&-l+\rmc s n
    \end{pmatrix}, 
  \end{equation}
  with a new imaginary unit \(\rmc\) and an additional parameter \(s\)
  usually equal to \(\pm 1\). The values of \(\bs:=\rmc^2\) is \(-1\),
  \(0\) or \(1\) independently from the value of \(\sigma\).  The
  matrix~\eqref{eq:FSCc-matrix} is the cornerstone of (extended)
  Fillmore--Springer--Cnops construction (FSCc)~\cite{Cnops02a} and
  closely related to technique recently used by A.A.~Kirillov to study
  the Apollonian gasket~\cite{Kirillov06}.

  The significance of FSCc in Erlangen framework is provided by the
  following result:
  \begin{thm}
    \label{th:FSCc-intertwine}
    The image  \(\tilde{C}_{\bs}^s\) of a cycle \(C_{\bs}^s\) under
    transformation~\eqref{eq:moebius} with \(g\in\SL\) is given by
    similarity of the matrix~\eqref{eq:FSCc-matrix}:
    \begin{equation}
      \label{eq:cycle-similarity}
      \tilde{C}_{\bs}^s= gC_{\bs}^sg^{-1}.
    \end{equation}
    In other words FSCc~\eqref{eq:FSCc-matrix} \emph{intertwines}
    M\"obius action~\eqref{eq:moebius} on cycles with
    linear map~\eqref{eq:cycle-similarity}.
  \end{thm}
  
  There are several ways to prove~\eqref{eq:cycle-similarity}: either
  by a brute force calculation (fortunately
  \href{http://arxiv.org/abs/cs.MS/0512073}{performed by a
    CAS})~\cite{Kisil05a} or through the related orthogonality of
  cycles~\cite{Cnops02a}, see the end of the next
  section~\ref{sec:invar-algebr-geom}. 

  The important observation here is that FSCc~\eqref{eq:FSCc-matrix}
  uses an imaginary unit \(\rmc\) which is not related to \(\rmi\)
  defining the appearance of cycles on plane. In other words any EPH
  type of geometry in the cycle space \(\Space{P}{3}\) admits drawing
  of cycles in the point space \(\Space{R}{2}\) as circles, parabolas
  or hyperbolas. We may think on points of \(\Space{P}{3}\) as ideal
  cycles while their depictions on \(\Space{R}{2}\) are only their
  shadows on the wall of
  \href{http://en.wikipedia.org/wiki/Plato#Metaphysics}{Plato's
    cave}. 

\begin{figure}[htbp]
  \centering
  (a) \includegraphics[scale=0.8]{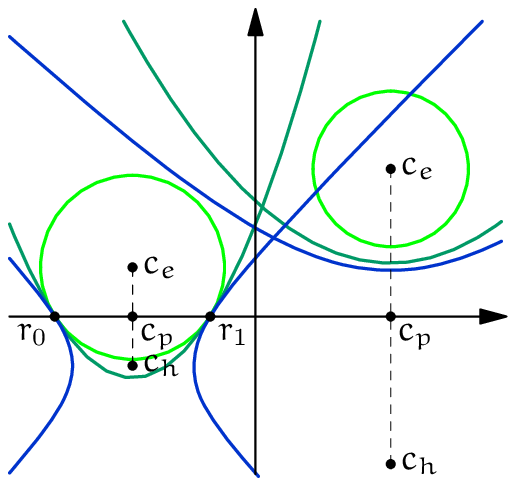}\hspace{1cm}
  (b) \includegraphics[scale=0.8]{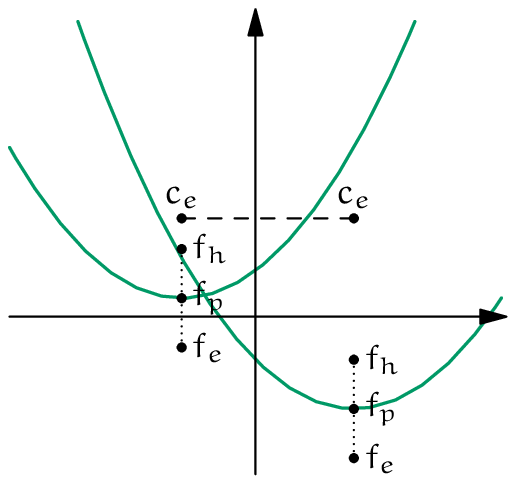}
  \caption[Cycle implementations, centres and foci]{\small
    (a) Different
    EPH implementations of the same cycles defined by quadruples of
    numbers.\\
    (b) Centres and foci of two parabolas with the same focal length.} 
  \label{fig:eph-cycle}
\end{figure}
  Fig.~\ref{fig:eph-cycle}(a) shows the same cycles drawn
  in different EPH styles. Points \(c_{e,p,h}=(\frac{l}{k}, -\sigma
  \frac{n}{k})\) are their respective e/p/h-centres. They are related to
  each other through several identities:
  \begin{equation}
    \label{eq:centres}
    c_e=\bar{c}_h, \quad c_p=\frac{1}{2}(c_e+c_h).
  \end{equation}
  Fig.~\ref{fig:eph-cycle}(b) presents two cycles drawn as parabolas,
  they have the same focal length \(\frac{n}{2k}\) and thus their
  e-centres are on the same level. In other words \emph{concentric}
  parabolas are obtained by a vertical shift, not
  scaling as an analogy with circles or hyperbolas may suggest. 

  Fig.~\ref{fig:eph-cycle}(b) also presents points, called e/p/h-foci:
  \begin{equation}
    \label{eq:foci}
    f_{e,p,h}=\left(\frac{l}{k}, -\frac{\det C_{\bs}^s}{2nk}\right),
  \end{equation}
  which are independent of the sign of \(s\).  If a cycle is depicted
  as a parabola then h-focus, p-focus, e-focus are correspondingly
  geometrical focus of the parabola, its vertex, and the point on the
  directrix nearest to the vertex.

  As we will see, cf. Thms.~\ref{th:ghost1} and~\ref{th:ghost2}, all
  three centres and three foci are useful  
  attributes of a cycle even if it is drawn as a circle.

  \subsection{Invariants: algebraic and geometric}
  \label{sec:invar-algebr-geom}

  We use known algebraic invariants of matrices to build appropriate
  geometric invariants of cycles. It is yet another demonstration that
  any division of mathematics into subjects is only illusive.

  For \(2\times 2\) matrices (and thus cycles) there are only two
  essentially different invariants under
  similarity~\eqref{eq:cycle-similarity} (and thus under M\"obius
  action~\eqref{eq:moebius}): the \emph{trace} and the
  \emph{determinant}.  The latter was already used in~\eqref{eq:foci}
  to define cycle's foci. However due to projective nature of the
  cycle space \(\Space{P}{3}\) the absolute values of trace or
  determinant are irrelevant, unless they are zero.

  Alternatively we may have a special arrangement for normalisation of
  quadruples \((k,l,n,m)\). For example, if \(k\neq0\) we may
  normalise the quadruple to
  \((1,\frac{l}{k},\frac{n}{k},\frac{m}{k})\) with highlighted cycle's
  centre. Moreover in this case \(\det \cycle{s}{\bs}\) is equal to
  the square of cycle's radius, cf. Section~\ref{sec:dist-lenght-perp}.
  Another normalisation \(\det \cycle{s}{\bs}=1\) is used
  in~\cite{Kirillov06} to get a nice condition for touching circles.

  We still get important characterisation even with non-normalised
  cycles, e.g., invariant classes (for different \(\bs\)) of
  cycles are defined by the condition \(\det C_{\bs}^s=0\). Such a
  class is parametrises only by two real number and as such is easily
  attached to certain point of \(\Space{R}{2}\). For example, the
  cycle \(C_{\bs}^s\) with \(\det C_{\bs}^s=0\), \(\bs=-1\) drawn
  elliptically represent just a point \((\frac{l}{k},\frac{n}{k})\),
  i.e. (elliptic) zero-radius circle.  The same condition with
  \(\bs=1\) in hyperbolic drawing produces a null-cone originated at
  point \((\frac{l}{k},\frac{n}{k})\):
  \begin{displaymath}
    (u-\frac{l}{k})^2-(v-\frac{n}{k})^2=0,
  \end{displaymath}
  i.e. a zero-radius cycle in hyperbolic metric. 

\begin{figure}[htbp]
  \centering
  \includegraphics[scale=.6]{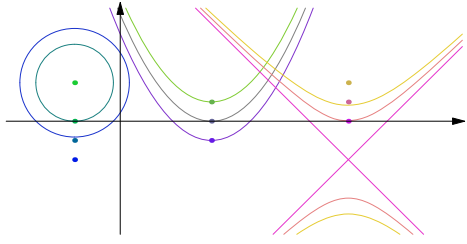}
  \caption[Different implementations of the same
    zero-radius cycles]{\small Different \(\rmi\)-implementations of the same
    \(\bs\)-zero-radius cycles and corresponding foci.}
  \label{fig:zero-radius}
\end{figure}
  In general for every notion there is nine possibilities: three EPH
  cases in the cycle space times three EPH realisations in the point
  space. Such nine cases for ``zero radius'' cycles is shown on
  Fig.~\ref{fig:zero-radius}. For example, p-zero-radius cycles in any
  implementation touch the real axis.

  This ``touching'' property is a manifestation of the \emph{boundary
    effect} in the upper-half plane
  geometry~\cite{Kisil05a}*{Rem.~3.4}. The famous question
  \href{http://en.wikipedia.org/wiki/Hearing_the_shape_of_a_drum}{on
    hearing drum's shape} has a sister:
  \begin{quote}
  \emph{Can we see/feel the boundary from inside a domain?}
  \end{quote}
  Both orthogonality relations  described below are ``boundary
  aware'' as well. It is not surprising after all since \(\SL\) action
  on the upper-half plane was obtained as an extension of its
  action~\eqref{eq:moebius} on the boundary. 

  According to the \wiki{Category_theory}{categorical viewpoint}
  internal properties of objects are of minor importance in comparison
  to their relations with other objects from the same class. 
Thus from now on we will
  look for invariant relations between two or more cycles.

  \subsection{Joint invariants: orthogonality}
  \label{sec:joint-invar-orth}

  The most expected relation between cycles is based on the following
  M\"obius invariant ``inner product'' build from a trace of
  product of two cycles as matrices:
  \begin{equation}
    \label{eq:inner-prod}
    \scalar{C_{\bs}^s}{\tilde{C}_{\bs}^s}= \tr (C_{\bs}^s\tilde{C}_{\bs}^s)
  \end{equation}
  By the way, an inner product of this type is used, for example, in
  \wiki{Gelfand-Naimark-Segal_construction}{GNS construction} to make
  a Hilbert space out of \(C^*\)-algebra.  The next standard move is
  given by the following definition.
  \begin{defn}
    \label{de:orthogonality}
    Two cycles are called \(\bs\)-orthogonal if
    \(\scalar{C_{\bs}^s}{\tilde{C}_{\bs}^s}=0\). 
  \end{defn}
  For the case of \(\bs \sigma=1\), i.e. when geometries of the cycle
  and point spaces are both either elliptic or hyperbolic, such an
  orthogonality is the standard one, defined in terms of angles
  between tangent lines in the intersection points of two cycles.
  However in the remaining seven (\(=9-2\)) cases the innocent-looking
  Defn.~\ref{de:orthogonality} brings unexpected relations.

\begin{figure}[htbp]
  \includegraphics[scale=.85]{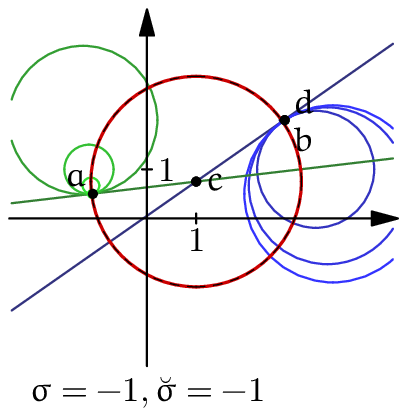}\hspace{1cm}
  \includegraphics[scale=.85]{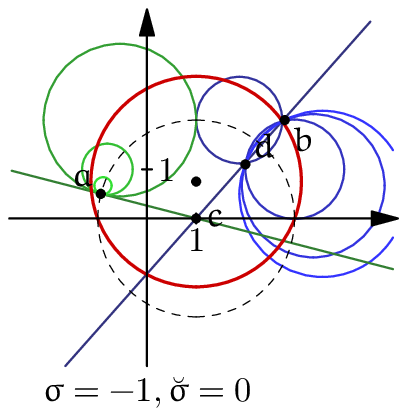}\hspace{1cm}
  \includegraphics[scale=.85]{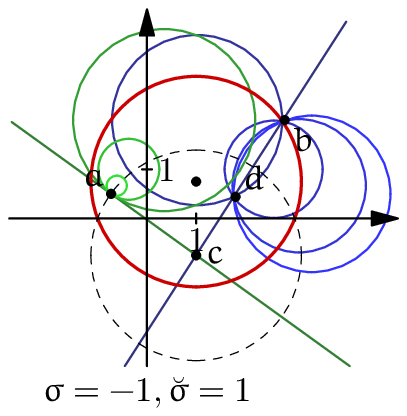}  \caption[Orthogonality of the
  first kind]{\small Orthogonality of the first
    kind in the elliptic point space.\\
    Each picture presents two groups (green and blue) of cycles which
    are orthogonal to the red cycle \(C^{s}_{\bs}\).  Point \(b\)
    belongs to \(C^{s}_{\bs}\) and the family of blue cycles
    passing through \(b\) is orthogonal to \(C^{s}_{\bs}\). They
    all also intersect in the point \(d\) which is the inverse of
    \(b\) in \(C^{s}_{\bs}\). Any orthogonality is reduced to the usual
    orthogonality with a new (``ghost'') cycle (shown by the dashed
    line), which may or may not coincide with \(C^{s}_{\bs}\). For
    any point \(a\) on the ``ghost'' cycle the orthogonality is
    reduced to the local notion in the terms of tangent lines at the
    intersection point. Consequently such a point \(a\) is always the
    inverse of itself.}
  \label{fig:orthogonality1}
\end{figure}
  Elliptic (in the point space) realisations of
  Defn.~\ref{de:orthogonality}, i.e. \(\sigma=-1\) is shown in
  Fig.~\ref{fig:orthogonality1}. The left picture corresponds to the
  elliptic cycle space, e.g. \(\bs=-1\). The orthogonality between
  the red circle and any circle from the blue or green families is
  given in the usual Euclidean sense. The central (parabolic in the
  cycle space) and the right (hyperbolic) pictures show non-local
  nature of the orthogonality.  There are analogues pictures in
  parabolic and hyperbolic point spaces as well~\cite{Kisil05a}.

  This orthogonality may still be expressed in the traditional sense
  if we will associate to the red circle the corresponding ``ghost''
  circle, which shown by the dashed line in Fig.~\ref{fig:orthogonality1}.
  To describe ghost cycle we need the
  \wiki{Heaviside_step_function}{\emph{Heaviside function}} \(\chi(\sigma)\):
  \begin{equation}
    \label{eq:heaviside-function}
    \chi(t)=\left\{
      \begin{array}{ll}
        1,& t\geq 0;\\
        -1,& t<0.
      \end{array}\right.
  \end{equation}

  \begin{thm}
    \label{th:ghost1}
    A cycle is \(\bs\)-orthogonal to cycle \(C_{\bs}^s\) if it is
    orthogonal in the usual sense to the \(\sigma\)-realisation of
    ``ghost'' cycle \(\hat{C}_{\bs}^s\), which is defined by the
    following two conditions:
    \begin{enumerate}
    \item \label{item:centre-centre-rel}
      \(\chi(\sigma)\)-centre of \(\hat{C}_{\bs}^s\) coincides
      with  \(\bs\)-centre of \(C_{\bs}^s\).
    \item Cycles \(\hat{C}_{\bs}^s\) and \(C^{s}_{\bs}\) have the same
      roots, moreover \(\det \hat{C}_{\sigma}^1= \det C^{\chi(\bs)}_{\sigma}\).
    \end{enumerate}
  \end{thm}
  The above connection between various centres of cycles illustrates
  their meaningfulness within our approach.

  One can easy check the following orthogonality properties of the
  zero-radius cycles defined in the previous section:
  \begin{enumerate}
  \item Since \(\scalar{C_{\bs}^s}{{C}_{\bs}^s}=\det {C}_{\bs}^s\)
    zero-radius cycles are  self-orthogonal (isotropic) ones.
  \item \label{it:ortho-incidence}
    A cycle \(\cycle{s}{\bs}\) is \(\sigma\)-orthogonal to a zero-radius
    cycle \(\zcycle{s}{\bs}\) if and only if \(\cycle{s}{\bs}\) passes
    through the \(\sigma\)-centre of \(\zcycle{s}{\bs}\).
  \end{enumerate}

  \subsection{Higher order joint invariants: s-orthogonality}
  \label{sec:higher-order-joint}

  With appetite already wet one may wish to build more joint
  invariants. Indeed for any homogeneous polynomial
  \(p(x_1,x_2,\ldots,x_n)\) of several non-commuting variables one may
  define an invariant joint disposition of \(n\) cycles
  \({}^j\!\cycle{s}{\bs}\) by the condition:
  \begin{displaymath}
    \tr p({}^1\!\cycle{s}{\bs}, {}^2\!\cycle{s}{\bs}, \ldots,  {}^n\!\cycle{s}{\bs})=0.
  \end{displaymath}
  However it is preferable to keep some geometrical meaning of
  constructed notions.

  An interesting observation is that in the matrix similarity of
  cycles~\eqref{eq:cycle-similarity} one may replace element
  \(g\in\SL\) by an arbitrary matrix corresponding to another cycle.
  More precisely the product
  \(\cycle{s}{\bs}\cycle[\tilde]{s}{\bs}\cycle{s}{\bs}\) is again the
  matrix of the form~\eqref{eq:FSCc-matrix} and thus may be associated
  to a cycle. This cycle may be considered as the reflection of
  \(\cycle[\tilde]{s}{\bs}\) in \(\cycle{s}{\bs}\).
  \begin{defn}
    \label{de:s-ortho}
    A cycle \(\cycle{s}{\bs}\) is s-orthogonal \emph{to} a cycle
    \(\cycle[\tilde]{s}{\bs}\) if the reflection of
    \(\cycle[\tilde]{s}{\bs}\) in \(\cycle{s}{\bs}\) is orthogonal
    (in the sense of Defn.~\ref{de:orthogonality}) to the real line.
    Analytically this is defined by:
    \begin{equation}
      \label{eq:s-orthog-def}
      \tr(\cycle{s}{\bs} \cycle[\tilde]{s}{\bs}\cycle{s}{\bs}\realline{s}{\bs})=0.
    \end{equation}
  \end{defn}
  Due to invariance of all components in the above definition
  s-orthogonality is a M\"obius invariant condition. Clearly
  this is not a symmetric relation: if \(\cycle{s}{\bs}\) is s-orthogonal to
  \(\cycle[\tilde]{s}{\bs}\) then \(\cycle[\tilde]{s}{\bs}\) is not
  necessarily s-orthogonal to
  \(\cycle{s}{\bs}\).
  
\begin{figure}[htbp]
  \includegraphics[scale=.85]{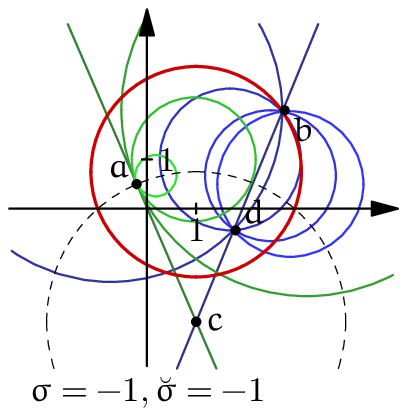}\hspace{1cm}
  \includegraphics[scale=.85]{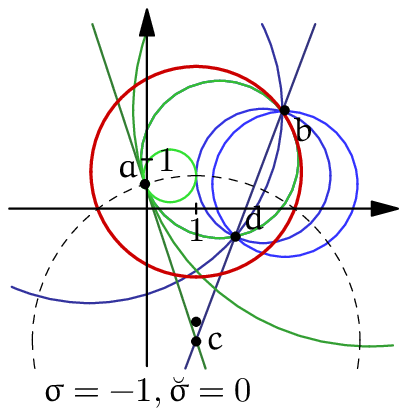}\hspace{1cm}
  \includegraphics[scale=.85]{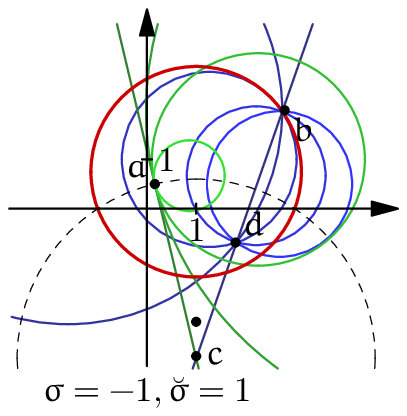}
  \caption[Orthogonality of the second kind]{\small Orthogonality of the
    second kind for circles. To highlight both
    similarities and distinctions with the ordinary orthogonality we
    use the same notations as that in Fig.~\ref{fig:orthogonality1}.}
  \label{fig:orthogonality2}
\end{figure}

  Fig.~\ref{fig:orthogonality2} illustrates s-orthogonality in the
  elliptic point space. By contrast with Fig.~\ref{fig:orthogonality1}
  it is not a local notion at the intersection points of cycles
  for all \(\bs\). However it may be again clarified in terms of the
  appropriate s-ghost cycle, cf. Thm.~\ref{th:ghost1}.
  \begin{thm}
    \label{th:ghost2}
    A cycle is s-orthogonal to a cycle \(C^{s}_{\bs}\) if its
    orthogonal in the traditional sense to its \emph{s-ghost cycle}
    \(\cycle[\tilde]{\bs}{\bs} = \cycle{\chi(\sigma)}{\bs}
    \Space[\bs]{R}{\bs} \cycle{\chi(\sigma)}{\bs}\), which is the
    reflection of the real line in \(\cycle{\chi(\sigma)}{\bs}\) and
    \(\chi\) is the \emph{Heaviside
      function}~\eqref{eq:heaviside-function}.  Moreover
    \begin{enumerate}
    \item \label{item:focal-centre-rel} \(\chi(\sigma)\)-Centre of
      \(\cycle[\tilde]{\bs}{\bs}\) coincides with the \(\bs\)-focus of
      \(\cycle{s}{\bs}\), consequently all lines s-orthogonal to
      \(\cycle{s}{\bs}\) are passing the respective focus.
    \item Cycles \(\cycle{s}{\bs}\) and \(\cycle[\tilde]{\bs}{\bs}\)
      have the same roots.
    \end{enumerate}
  \end{thm}
  Note the above intriguing interplay between cycle's centres
  and foci. Although s-orthogonality may look
  exotic it will naturally appear in the end of next Section again.

  Of course, it is possible to define another interesting higher order joint
  invariants of two or even more cycles.
  
  \subsection{Distance, length and perpendicularity}
  \label{sec:dist-lenght-perp}
  Geo\emph{metry} in the plain meaning of this word deals with \emph{distances}
  and \emph{lengths}. Can we obtain them from cycles?
  
\begin{figure}[htbp]
  \centering
  (a) \includegraphics[scale=.75]{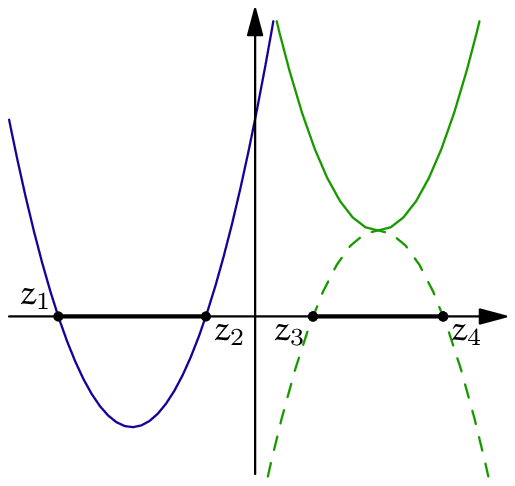}\hfill
  (b) \includegraphics[scale=.75]{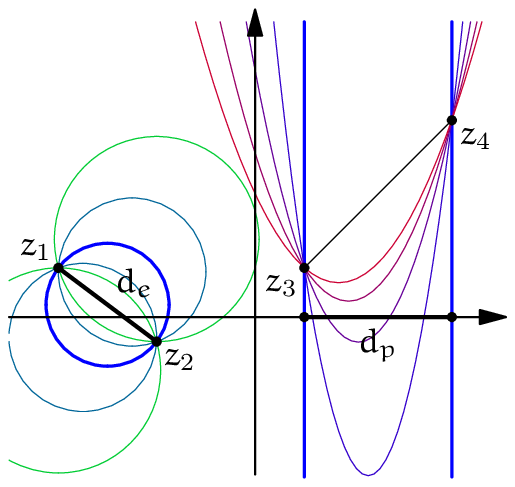}\hfill
  (c) \includegraphics[scale=.75]{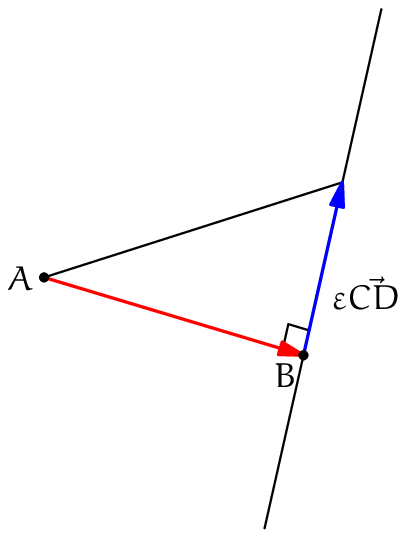}
  \caption[Radius and distance]{\small (a) The square of the parabolic
    diameter is the square of the distance between roots if they are
    real (\(z_1\) and \(z_2\)), otherwise the negative square of the
    distance between the adjoint roots (\(z_3\) and \(z_4\)).\\
    (b) Distance as extremum of diameters in elliptic (\(z_1\) and
    \(z_2\)) and parabolic (\(z_3\) and \(z_4\)) cases.\\
    (c) Perpendicular as the shortest route to a line.}
  \label{fig:distances}
\end{figure}

  We mentioned already that for circles normalised by the condition
  \(k=1\) the value \(\det
  \cycle{s}{\bs}=\scalar{\cycle{s}{\bs}}{\cycle{s}{\bs}}\) produces
  the square of the traditional circle radius. Thus we may keep it as the
  definition of the \emph{radius} for any cycle. But then we need to
  accept that in the parabolic case the radius is the (Euclidean) distance
  between (real) roots of the parabola, see
  Fig.~\ref{fig:distances}(a).

  Having radii of circles already defined we may use them for other
  measurements in several different ways. For example, the following
  variational definition may be used:

  \begin{defn}
    \label{de:distance}
    The \emph{distance} between two points is the extremum of
    diameters of all cycles passing through both points, see
    Fig.~\ref{fig:distances}(b).
  \end{defn}
  
  If \(\bs=\sigma\) this definition gives in all EPH cases the
  distance between endpoints of a vector \(z=u+\rmi v\) as follows:
  \begin{equation}
    \label{eq:eph-distance}
    d_{e,p,h}(u,v)^2=(u+\rmi v)(u-\rmi v)=u^2-\sigma  v^2.
  \end{equation}
  The parabolic distance \(d_p^2=u^2\), see
  Fig.~\ref{fig:distances}(b), algebraically sits between \(d_e\) and
  \(d_h\) according to the general principle~\eqref{eq:eph-class} and
  is widely accepted~\cite{Yaglom79}. However one may be unsatisfied
  by its degeneracy.

  An alternative measurement is motivated by the fact that a circle is
  the set of equidistant points from its centre. However the choice of
  ``centre'' is now rich: it may be either point from three
  centres~\eqref{eq:centres} or three foci~\eqref{eq:foci}.
  \begin{defn}
    \label{de:length}
    The \emph{length} of a directed interval \(\lvec{AB}\) is the radius
    of the cycle with its \emph{centre} (denoted by \(l_c(\lvec{AB})\))
    or \emph{focus} (denoted by \(l_f(\lvec{AB})\)) at the point \(A\)
    which passes through \(B\). 
  \end{defn}

  These definition is less common and have some unusual properties
  like non-symmetry: \(l_f(\lvec{AB})\neq l_f(\lvec{BA})\). However it
  comfortably fits the Erlangen program due to its
  \(\SL\)-\emph{conformal invariance}:

  \begin{thm}[\cite{Kisil05a}]
    Let \(l\) denote either the EPH distances~\eqref{eq:eph-distance}
    or any length from Defn.~\ref{de:length}. Then for
    fixed \(y\), \(y'\in\Space{R}{\sigma}\) the limit:
    \begin{displaymath}
      \lim_{t\rightarrow 0} \frac{l(g\cdot y, g\cdot(y+ty'))}{l(y,
        y+ty')}, \qquad
      \text{ where } g\in\SL, 
    \end{displaymath}
    exists and its value depends only from \(y\) and \(g\) and is
    independent from \(y'\).
  \end{thm}
  
  We may return from distances to angles recalling that in the
  Euclidean space a perpendicular provides the shortest root from a
  point to a line, see Fig.~\ref{fig:distances}(c). 
  \begin{defn}
    \label{de:perpendicular}
    Let \(l\) be a length or distance.  We say that a vector \(\lvec{AB}\) is
    \emph{\(l\)-perpendicular} to a vector \(\lvec{CD}\) if function
    \(l(\lvec{AB}+\varepsilon \lvec{CD})\) of a variable \(\varepsilon\) has a
    local extremum at \(\varepsilon=0\). 
  \end{defn}
  A pleasant surprise is that \(l_f\)-perpendicularity obtained
  thought the length from focus (Defn.~\ref{de:length}) coincides with
  already defined in Section~\ref{sec:higher-order-joint}
  s-orthogonality as follows from
  Thm.~\ref{th:ghost2}(\ref{item:focal-centre-rel}). It is
  also possible~\cite{Kisil08a} to make \(\SL\) action isometric in
  all three cases. 

  All these study are waiting to be  generalised to high dimensions
  and \wiki{Clifford_algebra}{Clifford algebras} provide a suitable
  language for this~\cite{Kisil05a
  }. 

\section{Analytic Functions}
\label{sec:analytic-functions}

We saw in the previous section that an inspiring geometry of cycles
can be recovered from the properties of \(\SL\).
In this section we consider a realisation of the function theory within
Erlangen approach~\cite{Kisil97c,Kisil97a,Kisil01a,Kisil02c}.

\subsection{Wavelet Transform and Cauchy Kernel}
\label{sec:backgr-compl-analys}

Elements of \(\SL\) could be also represented by \(2\times 2\)-matrices
with complex entries such that:
\begin{displaymath}
  g= \matr{\alpha}{\Bb}{\beta}{\Ba},
  \qquad 
  g^{-1}= \matr{\Ba}{-\Bb}{-\beta}{\alpha}, 
  \qquad
  \modulus{\alpha}^2-\modulus{\beta}^2=1.
\end{displaymath}
This realisations of \(\SL\) (or rather \(SU(2,\Space{C}{})\)) is more
suitable for function theory in the unit disk. It is obtained from the
form, which we used before for the upper half-plane, by means of the
Cayley transform~\cite[\S~8.1]{Kisil05a}.

We may identify the unit disk \(\Space{D}{}\) with the homogeneous space
\(\SL/\Space{T}{}\) for the unit circle \(\Space{T}{}\) through
the important decomposition \(\SL\sim \Space{D}{}\times\Space{T}{}\)\
with \(K=\Space{T}{}\)---the only compact subgroup of \(\SL\):
\begin{eqnarray}
\label{eq:sl2-u-psi-coord}
  \matr{\alpha}{\bar{\beta}}{\beta}{\bar{\alpha}} 
  & =& \modulus{\alpha} 
  \matr{1}{\bar{\beta}\bar{\alpha}^{-1}}{{\beta}{\alpha}^{-1}}{1}
  \matr{ 
    \frac{{\alpha}}{ \modulus{\alpha} } }{0}{0}{\frac{\bar{\alpha}}{ 
      \modulus{\alpha} } }
  \nonumber \\
  &=& \frac{1}{\sqrt{1- \modulus{u}^2 }}
  \matr{1}{u}{\bar{u}}{1}, \nonumber 
  \matr{e^{i\omega}}{0}{0}{e^{-i\omega}}
  \end{eqnarray}
where
\begin{displaymath}
  \omega=\arg \alpha,\qquad 
  u=\bar{\beta}\bar{\alpha}^{-1},\qquad \modulus{u}<1.
\end{displaymath}
Each element \(g\in\SL\) acts by the linear-fractional transformation
(the M\"obius map) on \(\Space{D}{}\)\ 
and \(\Space{T}{}\)
\(\FSpace{H}{2}(\Space{T}{})\) as follows: 
\begin{equation}
  \label{eq:moebius-su}
  g^{-1}: z \mapsto \frac{\Ba z - \Bb}{\alpha-{\beta} z},
\qquad \textrm{ where } \quad
g^{-1}=\matr{\Ba}{-\Bb}{-\beta}{\alpha}.
\end{equation}
In the decomposition~\eqref{eq:sl2-u-psi-coord} the first matrix on
the right hand side acts by transformation~\eqref{eq:moebius} as an
orthogonal rotation of \(\Space{T}{}\) or \(\Space{D}{}\); and the
second one---by transitive family of maps of the unit disk onto
itself.

The standard linearisation procedure~\cite[\S~7.1]{Kirillov76} leads
from M\"obius transformations~\eqref{eq:moebius} to the unitary
representation \(\rho_1\) irreducible on the \emph{Hardy space}:
\begin{equation}
  \label{eq:rho-1-1}
  \rho_1(g): f(z) \mapsto
  \frac{1}{\alpha-{\beta}{z}} \,
  f\left(
    \frac{\Ba z - \Bb}{\alpha-{\beta} z} 
  \right)
\qquad  \textrm{ where } \quad 
g^{-1}=\matr{\Ba}{-\Bb}{-\beta}{\alpha}.
\end{equation}
M\"obius transformations provide a natural family of
intertwining operators for \(\rho_1\) coming from inner
automorphisms of \(\SL\) (will be used later). 
 
We choose~\cite{Kisil98a,Kisil01a} \(K\)-invariant function \(v_0(z)\equiv 1\) 
to be  a \emph{vacuum vector}.
Thus the associated \emph{coherent states}
\begin{displaymath}
  v(g,z)=\rho_1(g)v_0(z)= (u-z)^{-1}
\end{displaymath} 
are completely determined by the point on the unit disk \(
u=\bar{\beta}\bar{\alpha}^{-1}\). The family of coherent states considered
as a function of both \(u\) and \(z\) is obviously the \textit{Cauchy
  kernel}~\cite{Kisil97c}. The \emph{wavelet transform}~\cite{Kisil97c,Kisil98a}
\(\oper{W}:\FSpace{L}{2}(\Space{T}{})\rightarrow
\FSpace{H}{2}(\Space{D}{}): f(z)\mapsto
\oper{W}f(g)=\scalar{f}{v_g}\)\ is the \textit{Cauchy integral}:
\begin{equation}
  \label{eq:cauchy}
  \oper{W} f(u)=\frac{1}{2\pi i}\int_{\Space{T}{}}f(z)\frac{1}{u-z}\,dz.
\end{equation}

We start from the following observation reflected in the
almost any textbook on complex analysis:
\begin{prop}
\label{pr:function-theory}
\emph{Analytic function theory} in the unit disk \(\Space{D}{}\) is
a manifestation of 
the \textit{mock discrete series}
representation \(\rho_1\) of \(\SL\):
\begin{equation}
  \label{eq:rho-1}
  \rho_1(g): f(z) \mapsto
  \frac{1}{\alpha-{\beta}{z}} \,
  f\left(
    \frac{\Ba z - \Bb}{\alpha-{\beta} z} 
  \right), \quad \textup{ where }
  \matr{\Ba}{-\Bb}{-\beta}{\alpha}\in\SL.
\end{equation}
\end{prop}
Other classical objects of complex analysis (the Cauchy-Riemann
equation, the Taylor series, the Bergman space, etc.)  can be also
obtained~\cite{Kisil97c,Kisil01a} from representation \(\rho_1\) as
shown below.

\subsection{The Dirac (Cauchy-Riemann) and Laplace Operators}
Consideration of Lie groups is hardly possible without consideration of 
their Lie algebras, which are naturally represented by left and right
invariant vectors fields on groups. On a homogeneous space \(\Omega=G/H\) we
have also defined a left action of \(G\) and can be interested in left
invariant vector fields (first order differential operators). Due to the 
irreducibility of \( \FSpace{F}{2}( \Omega )\) under left action of \(G\) every
such vector field \(D\) restricted to \( \FSpace{F}{2}( \Omega )\) is a scalar
multiplier of identity \(D|_{\FSpace{F}{2}( \Omega )}=cI\). We are
in particular interested in the case \(c=0\).
\begin{defn} \cite{AtiyahSchmid80,KnappWallach76}
A \(G\)-invariant first order differential operator 
\begin{displaymath}
D_\tau: \FSpace{C}{\infty}(\Omega, \mathcal{S} \otimes V_\tau ) 
\rightarrow
\FSpace{C}{\infty}(\Omega, \mathcal{S} \otimes V_\tau )
\end{displaymath}
such that \(\oper{W}(\FSpace{F}{2}(X))\subset \object{ker} D_\tau\) is called 
\emph{(Cauchy-Riemann-)Dirac operator} on \(\Omega=G/H\)
associated with an irreducible representation \( \tau \) of \(H\) in a space 
\(V_\tau\) and a spinor bundle \(\mathcal{S}\).  
\end{defn}
The Dirac operator is explicitly defined by the 
formula~\cite[(3.1)]{KnappWallach76}:
\begin{equation} \label{eq:dirac-def}
D_\tau= \sum_{j=1}^n \rho(Y_j) \otimes c(Y_j) \otimes 1,
\end{equation}
where \(Y_j\) is an orthonormal basis of
\(\algebra{p}=\algebra{h}^\perp\)---the orthogonal completion of the Lie
algebra \(\algebra{h}\) of the subgroup \(H\) in the Lie algebra \(\algebra{g}\)
of \(G\); \(\rho(Y_j)\) is the infinitesimal generator of the right action of
\(G\) on \(\Omega\); \(c(Y_j)\) is Clifford multiplication by \(Y_i \in
\algebra{p}\) on the Clifford module \(\mathcal{S}\).  We also define 
an invariant Laplacian by the formula
\begin{equation} \label{eq:lap-def}
\Delta_\tau= \sum_{j=1}^n \rho(Y_j)^2 \otimes \epsilon_j \otimes 1,
\end{equation}
where \(\epsilon_j = c(Y_j)^2\) is \(+1\) or \(-1\).
\begin{prop}
Let all commutators of vectors of \( \algebra{h}^\perp \) belong to \(
\algebra{h}\), i.e.
\([\algebra{h}^\perp,\algebra{h}^\perp]\subset\algebra{h}\).  Let also \(f_0\)
be an eigenfunction for all vectors of \( \algebra{h} \) with eigenvalue \(0\)
and let also \(\oper{W}f_0\) be a null solution to the Dirac operator \(D\).
Then \(\Delta f(x)=0\) for all \(f(x)\in \FSpace{F}{2}(\Omega)\).
\end{prop}
\begin{proof}
Because \(\Delta\) is a linear operator and \( \FSpace{F}{2}(\Omega)\) is 
generated by \(\pi_0(s(a))\oper{W} f_0\) it is enough to check that \(
\Delta \pi_0(s(a))\oper{W} f_0=0 \). Because \(\Delta\)  and \(\pi_0\) commute 
it is enough to check that \(\Delta \oper{W} f_0=0\). Now we observe that
\begin{displaymath}
\Delta = D^2 - \sum_{i,j} \rho([Y_i,Y_j]) \otimes c(Y_i)c(Y_j) \otimes 1.
\end{displaymath}
Thus the desired assertion is follows from two identities
\(\rho([Y_i,Y_j])\oper{W}f_0=0\) for \([Y_i,Y_j]\in H\) and  \(D
\oper{W}f_0=0\).
\end{proof}
\begin{example}
Let \(G=\SL\) and \(H\) be its one-dimensional compact subgroup \(K\) generated by 
an element \(Z \in \algebra{sl}(2, \Space{R}{})\). Then 
\(\algebra{h}^\perp\) is spanned by two vectors \(Y_1=A\) and \(Y_2=B\). In such 
a situation we can use \( \Space{C}{} \) instead of the Clifford algebra. Then formula~\eqref{eq:dirac-def} takes a simple 
form \(D=r(A+iB)\). Infinitesimal action of this operator in the upper-half 
plane follows from calculation in~\cite[VI.5(8), IX.5(3)]{Lang85}, it is 
\([D_{ \Space{H}{} } f] (z)= -2i y \frac{ \partial f(z)}{ \partial \bar{z} }
\), \(z=x+iy\). Making the Caley transform we can find its action in the
unit disk \(D_{ \Space{D}{} } \): again the Cauchy-Riemann operator \( \frac{ 
\partial }{ \partial \bar{z} } \) is its principal component. 
We calculate  \(D_{ \Space{H}{} }\) explicitly now to stress the 
similarity with \(\Space{R}{1,1}\) case. 

For the upper half plane \(\Space{H}{}\) we have following formulas:
\begin{eqnarray*}
s&:&\Space{H}{} \rightarrow \SL: z=x+iy \mapsto 
g=\matr{y^{1/2}}{xy^{-1/2}}{0}{y^{-1/2}}; \nonumber \\
s^{-1}&:& \SL \rightarrow \Space{H}{}: \matr{a}{b}{c}{d} \mapsto z= 
\frac{ai+b}{ci+d};  \nonumber \\
\rho(g)&:& \Space{H}{} \rightarrow\Space{H}{} :z \mapsto s^{-1}( s(z) * g) 
\\
&& \qquad \qquad \qquad  =s^{-1}\matr{ay^{-1/2}+cxy^{-1/2}}{ 
by^{1/2}+dxy^{-1/2}}{cy^{-1/2}}{dy^{-1/2}}\\
&& \qquad \qquad \qquad =\frac{(yb+xd)+i(ay+cx)}{ci+d} 
\nonumber
\end{eqnarray*}
Thus the right action of \(\SL\) on \(\Space{H}{}\) is given by the formula
\begin{displaymath}
\rho(g)z=\frac{(yb+xd)+i(ay+cx)}{ci+d}= x+y \frac{bd+ac}{c^2+d^2}+iy 
\frac{1}{c^2+d^2}.
\end{displaymath}
For \(A\) and \(B\) in \(\algebra{sl}(2, \Space{R}{} )\) we have:
\begin{displaymath}
\rho(e^{At}) z = x+iy e^{2t}, \qquad \rho(e^{Bt}) z = x 
+y\frac{e^{2t}-e^{-2t}}{e^{2t}+e^{-2t}}+iy\frac{4}{e^{2t}+e^{-
2t}}.
\end{displaymath}
Thus
\begin{eqnarray*}
[\rho(A) f](z) & = & \frac{ \partial f (\rho(e^{At}) z )}{ \partial t }
|_{t=0} = 2y \partial_2 f(z), \\ {}
[\rho(B) f](z) & = & \frac{ \partial f (\rho(e^{Bt}) z )}{ \partial t } 
|_{t=0} = 2y \partial_1 f(z),
\end{eqnarray*}
where \(\partial_1\) and \(\partial_2\) are derivatives of \(f(z)\) with respect 
to real and imaginary party of \(z\) respectively. Thus we get 
\begin{displaymath}
D_{ \Space{H}{} }= i\rho(A) + \rho( B) = 2y i\partial_2 + 2y \partial_1=
2y \frac{ \partial }{ \partial \bar{z} }
\end{displaymath}
as was expected.
\end{example}

\subsection{The Taylor expansion}
For any decomposition \(f_a(x)=\sum_\alpha  \psi_\alpha(x) V_\alpha(a)\) 
of the coherent states \(f_a(x)\) by means of functions \(V_\alpha(a)\)
(where the sum can become eventually an integral) we have the
\emph{Taylor expansion} 
\begin{eqnarray} 
\widehat{f}(a) & = & \int_X f(x) \bar{f}_a(x)\, dx= \int_X f(x) \sum_\alpha 
\bar{\psi}_\alpha(x)\bar{V}_\alpha(a)\, dx  \nonumber \\
 & = &  \sum_\alpha 
\int_X f(x)\bar{\psi}_\alpha(x)\, dx \bar{V}_\alpha(a) \nonumber \\
 & = & \sum_{\alpha}^{\infty} \bar{V}_\alpha(a) f_\alpha,\label{eq:taylor}
\end{eqnarray}
where \(f_\alpha=\int_X f(x)\bar{\psi}_\alpha(x)\, dx\).
However to be useful within the presented scheme such a decomposition 
should be connected with the structures of \(G\), \(H\), and the representation 
\(\pi_0\). We will use a decomposition of \(f_a(x)\) by the eigenfunctions of 
the operators \(\pi_0(h)\), \(h\in \algebra{h}\).
\begin{defn}
 Let \(\FSpace{F}{2}=\int_{A} \FSpace{H}{\alpha}\,d\alpha\) be a spectral 
decomposition with respect to the operators \(\pi_0(h)\), \(h\in \algebra{h}\).
Then the decomposition
\begin{equation} \label{eq:spec-c}
 f_a(x)= \int_{A} V_\alpha(a) f_\alpha(x)\, d\alpha,
\end{equation}
where \(f_\alpha(x)\in \FSpace{H}{\alpha}\) and \(V_\alpha(a): 
\FSpace{H}{\alpha} \rightarrow \FSpace{H}{\alpha}\) is called the Taylor 
decomposition of the Cauchy kernel \(f_a(x)\).
\end{defn}
Note that the Dirac operator \(D\) is defined in the terms of left invariant 
shifts and therefor commutes with all \(\pi_0(h)\). Thus it also has a 
spectral decomposition over spectral subspaces of \(\pi_0(h)\):
\begin{equation} \label{eq:spec-d}
 D= \int_{A} D_\delta \, d\delta.
\end{equation}
We have obvious property
\begin{prop} \label{pr:cauchy-dirac}
If spectral measures \(d\alpha\) and \(d\delta\) 
from~\eqref{eq:spec-c} and~\eqref{eq:spec-d} have disjoint supports then 
the image of the Cauchy integral belongs to the kernel of the Dirac 
operator.
\end{prop}
For discrete series representation functions \(f_\alpha(x)\) can be 
found in \(\FSpace{F}{2}\) (as in Example~\ref{ex:taylor-a}), for the 
principal series representation this is not the case. To overcome confusion
one can think about the Fourier transform on the real line. It can be 
regarded as a continuous decomposition of a function \(f(x)\in 
\FSpace{L}{2}(\Space{R}{})\) over a set of harmonics \(e^{i\xi x}\) neither of
those belongs to \(\FSpace{L}{2}(\Space{R}{})\). This has a lot of common 
with the Example~3.10(b) in \cite{Kisil97c}.
\begin{example} \label{ex:taylor-a}
Let \(G=\SL\) and \(H=K\) be its maximal compact subgroup and
\(\pi_1\) defined in~\eqref{eq:rho-1-1}.  \(H\) acts on \(\Space{T}{}\) by
rotations.  It is one dimensional and eigenfunctions of its generator \(Z\)
are parametrized by integers (due to compactness of \(K\)).  Moreover, on the
irreducible Hardy space these are positive integers \(n=1,2,3\ldots\) and
corresponding eigenfunctions are \(f_n(\phi)=e^{i(n-1)\phi}\). Negative 
integers span the space of anti-holomorphic function and the splitting 
reflects the existence of analytic structure given by the Cauchy-Riemann 
equation. The decomposition of coherent states \(f_a(\phi)\) by means of this
functions is well known:
\begin{displaymath}
f_a(\phi)= \frac{ \sqrt[]{1- \modulus{a}^2 }}{ \bar{a}e^{i\phi}-1}= 
\sum_{n=1}^\infty \sqrt[]{1- \modulus{a}^2 }\bar{a}^{n-1} e^{i(n-1)\phi}=
\sum_{n=1}^\infty V_n(a)f_n(\phi),
\end{displaymath}
where \(V_n(a)=\sqrt[]{1- \modulus{a}^2 }\bar{a}^{n-1} \). This is the 
classical Taylor expansion up to multipliers coming from the invariant 
measure.
\end{example}

\section{Functional Calculus}
\label{sec:functional-calculus}

United in the trinity functional calculus, spectrum, and spectral
mapping theorem play the exceptional r\^ole in functional
analysis and could not be substituted by anything else. 
All traditional  definitions of functional calculus  are covered by the
following rigid template based on \emph{algebra homomorphism} property:
\begin{defn}
\label{de:calculus-old}
 An \emph{functional calculus} for an element
 \(a\in\algebra{A}\)\ is a continuous 
linear mapping
\(\Phi: \mathcal{ A}\rightarrow \algebra{A}\)\ such that
\begin{enumerate} 
\item 
 \(\Phi\)\ is a unital \emph{algebra homomorphism}
 \begin{displaymath}
   \Phi(f \cdot g)=\Phi(f) \cdot \Phi (g).
\end{displaymath}
\item 
 There is an initialisation condition: \(\Phi[v_0]=a\)\ for
 for a fixed function \(v_0\), e.g. \(v_0(z)=z\). 
\end{enumerate}
\end{defn}

Most typical definition of the spectrum is seemingly independent and 
uses the important notion of
resolvent: 
\begin{defn}
  \label{de:spectrum}
  A \emph{resolvent} of element \(a\in\algebra{A}\)\ is the function
  \(R(\lambda)=(a-\lambda e)^{-1}\), which is the image under
  \(\Phi\)\ of the Cauchy kernel \((z-\lambda)^{-1}\).

  A  \emph{spectrum} of \(a\in\algebra{A}\)\ is the set \(\spec a\)\ of
  singular points of its resolvent \(R(\lambda)\).
\end{defn}
Then the following important theorem links spectrum and functional calculus
together. 
\begin{thm}[Spectral Mapping]
  \label{th:spectral-mapping}
  For a function \(f\) suitable for the   functional calculus:
   \begin{equation}
     \label{eq:spectral-mapping}
     f(\spec a)=\spec  f(a).
   \end{equation}
\end{thm}

However the power of the classic spectral theory rapidly decreases if
we move beyond the study of one normal operator (e.g. for
quasinilpotent ones) and is virtually nil if we consider several
non-commuting ones.
Sometimes these severe limitations are seen to be irresistible and
alternative constructions, i.e. model theory~\cite{Nikolskii86}, were
developed.

Yet the spectral theory can be revived from a fresh start. While three
compon\-ents---functional calculus, spectrum, and spectral mapping
theorem---are highly interdependent in various ways 
we will nevertheless arrange them as follows: 

\begin{enumerate}
\item Functional  calculus is an \emph{original} notion defined in
  some independent terms;
\item Spectrum (or spectral decomposition) is derived from previously
  defined functional calculus as its \emph{support} (in some
  appropriate sense);
\item Spectral mapping theorem then should drop out naturally in the
  form~\eqref{eq:spectral-mapping} or some its variation.
\end{enumerate}

Thus the entire scheme depends from the notion of the functional
calculus and our ability to escape limitations of
Definition~\ref{de:calculus-old}.  The first known to the present
author definition of functional calculus not linked to algebra
homomorphism property was the Weyl functional calculus defined by an
integral formula~\cite{Anderson69}. Then its intertwining property
with affine transformations of Euclidean space was proved as a
theorem. However it seems to be the only ``non-homomorphism'' calculus
for decades.

The different approach to whole range of calculi was given
in~\cite{Kisil95i} and developed in~\cite{Kisil98a} in terms of
\emph{intertwining operators} for group representations. It was
initially targeted for several non-commuting operators because no
non-trivial algebra homomorphism with a commutative algebra of
function is possible in this case.  However it emerged later that the
new definition is a useful replacement for classical one across all
range of problems.

In the present note we will support the last claim by consideration of
the simple known problem: characterisation a \(n \times n\)\ matrix up
to similarity. Even that ``freshman'' question could be only sorted
out by the classical spectral theory for a small set of diagonalisable
matrices. Our solution in terms of new spectrum will be full and thus
unavoidably coincides with one given by the Jordan normal form of
matrix. Other more difficult questions are the subject of ongoing
research.

\subsection{Another Approach to Analytic Functional Calculus} 

Anything called ``\emph{functional} calculus'' uses properties of
\emph{functions} to model properties of \emph{operators}. Thus
changing our viewpoint on functions, as was done in
Section~\ref{sec:analytic-functions}, we could get another approach to
operators.

The representation~\eqref{eq:rho-1} is unitary irreducible when acts
on the Hardy space \(\FSpace{H}{2}\). Consequently we have one more
reason to abolish the template definition~\ref{de:calculus-old}:
\(\FSpace{H}{2}\) is \emph{not} an algebra. Instead we replace the
\textit{homomorphism  property} by a \textit{symmetric covariance}:
\begin{defn}
  \label{de:functional-calculus-new}
  An \emph{analytic functional calculus} for an element
  \(a\in\algebra{A}\)\ and  an
  \(\algebra{A}\)-module \(M\)\ is a \textit{continuous 
  linear} mapping
  \(\Phi:\FSpace{A}{}(\Space{D}{})\rightarrow \FSpace{A}{}(\Space{D}{},M)\)\ such that 
  \begin{enumerate} 
  \item \(\Phi\)\ is an \emph{intertwining operator} 
    \begin{displaymath}
      \Phi\rho_1=\rho_a\Phi
    \end{displaymath}
    between two representations of the
    \(\SL\)\ group \(\rho_1\)~\eqref{eq:rho-1} and \(\rho_a\)\
    defined below in~\eqref{eq:rho-a}.
  \item There is an initialisation condition: \(\Phi[v_0]=m\)\ for
    \(v_0(z)\equiv 1\) and \(m\in M\), where \(M\) is a left
    \(\algebra{A}\)-module.  
  \end{enumerate}
\end{defn} Note that our functional calculus released form the
homomorphism condition can take value in any left
\(\algebra{A}\)-module \(M\), which however could be \(\algebra{A}\)
itself if suitable. This add much flexibility to our construction.

The earliest functional calculus, which is \emph{not} an algebraic
homomorphism, was the Weyl functional calculus and
was defined just by an integral formula as an operator valued
distribution~\cite{Anderson69}. In that paper
(joint) spectrum was defined as support of the Weyl calculus, i.e. as
the set of point where this operator valued distribution does not
vanish. We also define
the spectrum as a support of functional calculus, but due to our
Definition~\ref{de:functional-calculus-new} it will means the set of
non-vanishing intertwining operators with primary subrepresentations.
\begin{defn}
  \label{de:spectrum-new}
    A corresponding \emph{spectrum} of \(a\in\algebra{A}\) is the
  \textit{support} of the functional
  calculus \(\Phi\), i.e. the collection of intertwining operators of
  \(\rho_a\) with \emph{prime representations}~\cite[\S~8.3]{Kirillov76}.
\end{defn}

More variations of functional calculi are obtained from other groups and their
representations~\cite{Kisil95i,Kisil98a}. 

\subsection[Representations in Banach Algebras]{Representations of $\SL$ in Banach Algebras}
A simple but important observation is that the M\"obius
transformations~\eqref{eq:moebius} can be easily extended to any 
Banach algebra.
  Let \(\algebra{A}\) be a Banach algebra with the unit \(e\), 
  an element \(a\in\algebra{A}\) with \(\norm{a}<1\) be fixed, then 
  \begin{equation}
    \label{eq:sl2-on-A}
    g: a \mapsto g\cdot a=(\Ba a -\Bb e)(\alpha e-\beta a)^{-1}, \qquad
    g\in\SL
  \end{equation}
  is a well defined \(\SL\) action on a subset \(\Space{A}{}=\{g\cdot
  a \such g\in 
  \SL\}\subset\algebra{A}\), i.e. \(\Space{A}{}\) is a \(\SL\)-homogeneous
  space. Let us define the \emph{resolvent} function
  \(R(g,a):\Space{A}{}\rightarrow \algebra{A}\):
  \begin{displaymath}
    R(g, a)=(\alpha e-\beta a)^{-1} \quad 
  \end{displaymath}
  then 
  \begin{equation}
    \label{eq:ind-rep-multipl}
    R(g_1,\ga)R(g_2,g_1^{-1}\ga)=R(g_1g_2,\ga).
  \end{equation}
  The last identity is well known in representation
  theory~\cite[\S~13.2(10)]{Kirillov76} and is a key ingredient of
  \emph{induced representations}. Thus we can again
  linearise~\eqref{eq:sl2-on-A} (cf.~\eqref{eq:rho-1-1}) in
  the space of continuous functions \(\FSpace{C}{}(\Space{A}{},M)\)
  with values in  a left
  \(\algebra{A}\)-module \(M\), e.g.\(M=\algebra{A}\):
  \begin{eqnarray}
    \rho_a(g_1): f(g^{-1}\cdot a ) &\mapsto&
    R(g_1^{-1}g^{-1}, a)f(g_1^{-1}g^{-1}\cdot a) \label{eq:rho-a}\\
    &&\quad =
    (\alpha' e-\beta'a)^{-1} \,
    f\left(
      \frac{\Ba' \cdot a - \Bb' e}{\alpha'  e -\beta' a} 
    \right).  \nonumber
  \end{eqnarray}  
  For any \(m\in M\) we can again define a \(K\)-invariant
  \emph{vacuum vector} as \(v_m(g^{-1}\cdot 
  a)=m\otimes v_0(g^{-1}\cdot a) \in \FSpace{C}{}(\Space{A}{},M)\). 
  It generates the associated with \(v_m\) family of \emph{coherent
    states} \(v_m(u,a)=(ue-a)^{-1}m\), where \(u\in\Space{D}{}\).

The \emph{wavelet transform}  defined by
the same common formula based on coherent states (cf.~\eqref{eq:cauchy}):
\[\oper{W}_m f(g)= \scalar{f}{\rho_a(g) v_m},\qquad \]
is a version of Cauchy integral, which maps
\(\FSpace{L}{2}(\Space{A}{})\) to \(\FSpace{C}{}(\SL,M)\). It is
 closely related (but not identical!) to the
Riesz-Dunford functional calculus:  the traditional functional
calculus is given by the case:
\begin{displaymath}
  \Phi: f \mapsto \oper{W}_m f(0) \qquad\textrm{ for } M=\algebra{A}
  \textrm{ and } m=e.
\end{displaymath}

The both conditions---the intertwining property and initial
value---required by Definition~\ref{de:functional-calculus-new} easily
follows from our construction.

\subsection[Jet Bundles and Prolongations]{Jet Bundles and Prolongations of $\rho_1$}
\label{sec:jet-bundl-prol-1}
Spectrum was defined in~\ref{de:spectrum-new} as
the \emph{support} of our functional calculus. To elaborate its meaning we
need the notion of a \emph{prolongation} of representations introduced by
S.~Lie, see  \cite{Olver93,Olver95} for a detailed exposition.

\begin{defn} \textup{\cite[Chap.~4]{Olver95}}
  Two holomorphic functions have \(n\)th \emph{order contact} in a point
  if their value and their first \(n\) derivatives agree at that point,
  in other words their Taylor expansions are the same in first \(n+1\)
  terms. 

  A point \((z,u^{(n)})=(z,u,u_1,\ldots,u_n)\) of the \emph{jet space}
  \(\Space{J}{n}\sim\Space{D}{}\times\Space{C}{n}\) is the equivalence
  class of holomorphic functions having \(n\)th contact at the point \(z\)
  with the polynomial:
  \begin{equation}\label{eq:Taylor-polynom}
    p_n(w)=u_n\frac{(w-z)^n}{n!}+\cdots+u_1\frac{(w-z)}{1!}+u.
  \end{equation}
\end{defn}

For a fixed \(n\) each holomorphic function
\(f:\Space{D}{}\rightarrow\Space{C}{}\) has \(n\)th \emph{prolongation}
(or \emph{\(n\)-jet}) \(\object[_n]{j}f: \Space{D}{} \rightarrow
\Space{C}{n+1}\): 
\begin{equation}\label{eq:n-jet}
  \object[_n]{j}f(z)=(f(z),f'(z),\ldots,f^{(n)}(z)).
\end{equation}The graph \(\Gamma^{(n)}_f\) of \(\object[_n]{j}f\) is a
submanifold of \(\Space{J}{n}\) which is section of the \emph{jet
bundle} over \(\Space{D}{}\) with a fibre \(\Space{C}{n+1}\). We also
introduce a notation \(J_n\) for the map \(
  J_n:f\mapsto\Gamma^{(n)}_f
\) of a holomorphic \(f\) to the graph \(\Gamma^{(n)}_f\) of its \(n\)-jet
\(\object[_n]{j}f(z)\)~\eqref{eq:n-jet}.

One can prolong any map of functions \(\psi: f(z)\mapsto [\psi f](z)\) to
a map \(\psi^{(n)}\) of \(n\)-jets by the formula
\begin{equation}\label{eq:prolong-def}
  \psi^{(n)} (J_n f) = J_n(\psi f).
\end{equation} For example such a prolongation \(\rho_1^{(n)}\) of the
representation \(\rho_1\) of the group \(\SL\) in
\(\FSpace{H}{2}(\Space{D}{})\) (as any other representation of a Lie
group~\cite{Olver95}) will be again a representation of
\(\SL\). Equivalently we can say that \(J_n\) \emph{intertwines} \(\rho_1\) and
\(\rho^{(n)}_1\):
\begin{displaymath}
   J_n \rho_1(g)= \rho_1^{(n)}(g) J_n \quad
  \textrm{ for all } g\in\SL.
\end{displaymath}
Of course, the representation \(\rho^{(n)}_1\) is not irreducible: any jet
subspace \(\Space{J}{k}\), \(0\leq k \leq n\) is
\(\rho^{(n)}_1\)-invariant subspace of \(\Space{J}{n}\).  However the
representations \(\rho^{(n)}_1\) are
\emph{primary}~\cite[\S~8.3]{Kirillov76} in the sense that they are not 
sums of two subrepresentations.

The following statement explains why jet spaces appeared in our study
of functional calculus.
\begin{prop}
  \label{pr:Jordan-zero}
  Let matrix \(a\) be a Jordan block of a length \(k\) with the
  eigenvalue \(\lambda=0\),
  and \(m\) be its root vector of order \(k\), i.e. \(a^{k-1}m\neq
  a^k m =0\). Then the restriction of \(\rho_a\) on the subspace
  generated by \(v_m\) is equivalent to the representation
  \(\rho_1^{k}\).
\end{prop}

\subsection{Spectrum and the Jordan Normal Form of a Matrix}
Now we are prepared to describe a spectrum of a matrix. Since the
functional calculus is an intertwining operator its support is
a decomposition into intertwining operators with prime
representations (we could not expect generally 
that these prime subrepresentations are irreducible).

Recall the transitive on \(\Space{D}{}\) group of inner
automorphisms of \(\SL\), which can send any
\(\lambda\in\Space{D}{}\) to \(0\) and are actually parametrised by
such a \(\lambda\).
This group extends Proposition~\ref{pr:Jordan-zero} to the complete
characterisation of \(\rho_a\) for matrices.
\begin{prop}  
  Representation \(\rho_a\) is equivalent to a direct sum of the
  prolongations \(\rho_1^{(k)}\) of \(\rho_1\) in the \(k\)th jet space
  \(\Space{J}{k}\) intertwined with inner automorphisms. Consequently
  the \textit{spectrum} of \(a\) (defined via the functional calculus
  \(\Phi=\oper{W}_m\)) labelled exactly by \(n\) pairs of numbers
  \((\lambda_i,k_i)\), \(\lambda_i\in\Space{D}{}\),
  \(k_i\in\Space[+]{Z}{}\), \(1\leq i \leq n\) some of whom could
  coincide.
\end{prop}
Obviously this spectral theory is a fancy restatement of the \emph{Jordan
  normal form} of matrices.

\begin{figure}[tb]
  \begin{center}
 (a) \includegraphics[scale=.8]{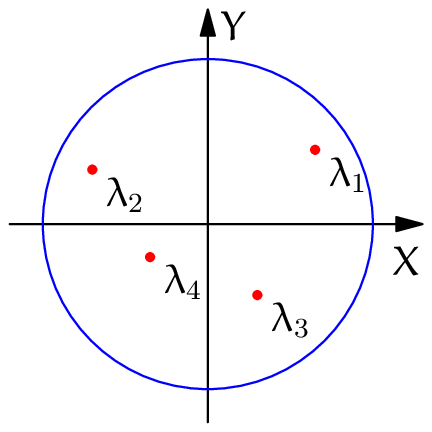}\hfill
  (b)\includegraphics[scale=.8]{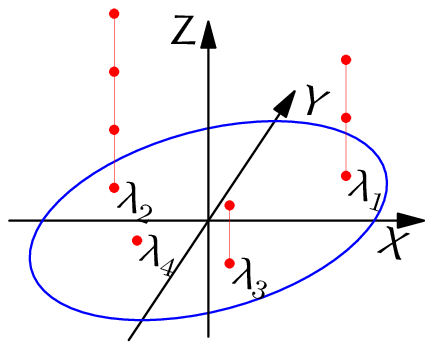}\hfill
  (c)\includegraphics[scale=.8]{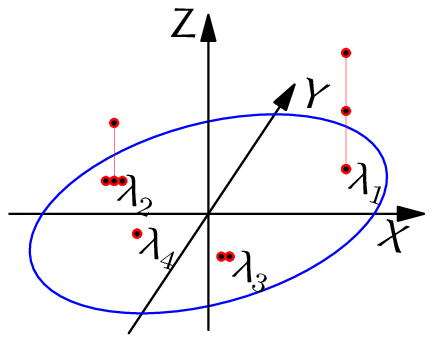}
 \caption[Three dimensional spectrum]{Classical spectrum of the matrix
   from the Ex.~\ref{ex:3dspectrum} is shown at (a). Covariant
   spectrum of the same matrix in the jet space is drawn at (b).  The
   image of the covariant spectrum under the map from
   Ex.~\ref{ex:spectral-mapping} is presented (c).} 
    \label{fig:3dspectrum}
  \end{center}
\end{figure}

\begin{example}
  \label{ex:3dspectrum}
  Let \(J_k(\lambda)\) denote the Jordan block of the length \(k\) for the
  eigenvalue \(\lambda\). On the Fig.~\ref{fig:3dspectrum} there are two
  pictures of the spectrum for the matrix
  \begin{displaymath}
    a=J_3\left(\lambda_1\right)\oplus     J_4\left(\lambda_2\right) 
    \oplus J_1\left(\lambda_3\right) \oplus      J_2\left(\lambda_4\right),
  \end{displaymath} 
  where
  \begin{displaymath}
    \lambda_1=\frac{3}{4}e^{i\pi/4}, \quad
    \lambda_2=\frac{2}{3}e^{i5\pi/6}, \quad
    \lambda_3=\frac{2}{5}e^{-i3\pi/4}, \quad
    \lambda_4=\frac{3}{5}e^{-i\pi/3}.
  \end{displaymath} Part (a) represents the conventional two-dimensional
  image of the spectrum, i.e. eigenvalues of \(a\), and
  \href{http://www.maths.leeds.ac.uk/~kisilv/calc1vr.gif}{(b) describes
  spectrum \(\spec{} a\) arising from the wavelet construction}. The
  first image did not allow to distinguish \(a\) from many other
  essentially different matrices, e.g. the diagonal matrix
  \begin{displaymath}
    \diag\left(\lambda_1,\lambda_2,\lambda_3,\lambda_4\right),
  \end{displaymath}
  which even have a different dimensionality.
  At the same time the Fig.~\ref{fig:3dspectrum}(b)
  completely characterise \(a\) up to a similarity. Note that each point of
  \(\spec a\) on Fig.~\ref{fig:3dspectrum}(b) corresponds to a particular
  root vector, which spans a primary subrepresentation.
\end{example}

\subsection{Spectral Mapping Theorem}
As was mentioned in the Introduction a resonable spectrum should be
linked to the corresponding functional calculus by an appropriate
spectral mapping theorem. The new version of spectrum is based on
prolongation of \(\rho_1\) into jet spaces (see
Section~\ref{sec:jet-bundl-prol-1}). Naturally a correct version of
spectral mapping theorem should also operate in jet spaces. 

Let \(\phi: \Space{D}{} \rightarrow \Space{D}{}\) be a holomorphic
map, let us define its action on functions \([\phi_*
f](z)=f(\phi(z))\). According to the general formula~\eqref{eq:prolong-def}
we can define the prolongation
\(\phi_*^{(n)}\) onto the jet space \(\Space{J}{n}\). Its associated
action \(\rho_1^k \phi_*^{(n)}=\phi_*^{(n)}\rho_1^n\) on the pairs
\((\lambda,k)\) is given by the formula:
\begin{equation}
  \label{eq:phi-star-action}
  \phi_*^{(n)}(\lambda,k)=\left(\phi(\lambda),
    \left[\frac{k}{\deg_\lambda \phi}\right]\right),
\end{equation}
where \(\deg_\lambda \phi\) denotes the degree of zero of the function
\(\phi(z)-\phi(\lambda)\) at the point \(z=\lambda\) and \([x]\) denotes
the integer part of \(x\). 

\begin{thm}[Spectral mapping] 
  Let \(\phi\) be a holomorphic mapping  \(\phi: \Space{D}{}
  \rightarrow \Space{D}{}\) and its prolonged action \(\phi_*^{(n)}\) defined
  by~\eqref{eq:phi-star-action}, then
  \begin{displaymath}
    \spec \phi(a) = \phi_*^{(n)} \spec a. 
  \end{displaymath}
\end{thm} 

The explicit expression of~\eqref{eq:phi-star-action} for
\(\phi_*^{(n)}\), which involves derivatives of \(\phi\) upto \(n\)th order,
is known, see for example~\cite[Thm.~6.2.25]{HornJohnson94}, but was
not recognised before as form of spectral mapping.

\begin{example}
  \label{ex:spectral-mapping}
  Let us continue with Example~\ref{ex:3dspectrum}. Let \(\phi\) map
  all four eigenvalues \(\lambda_1\), \ldots, \(\lambda_4\) of the
  matrix \(a\) into themselves. Then Fig.~\ref{fig:3dspectrum}(a) will
  represent the classical spectrum of \(\phi(a)\) as well as \(a\).

  However Fig.~\ref{fig:3dspectrum}(c) shows mapping of the new
  spectrum for the case
  \(\phi\)  has
  \textit{orders of zeros} at these points as follows: the order \(1\)
  at \(\lambda_1\), exactly the order \(3\) at \(\lambda_2\), an order
  at least \(2\) at \(\lambda_3\), and finally any order at
  \(\lambda_4\).
\end{example}


\section{Open Problems}
\label{sec:open-problems}
In this section we indicate several directions for further work which
go through three main areas described in the paper.. 
\makeatletter
\def\p@enumi{\thesubsection.}
\makeatother

\subsection{Geometry}
\label{sec:geometry-problems}
Geometry is most elaborated area so far, yet many directions are
waiting for further exploration. 
\begin{enumerate}
\item M\"obius transformations~\eqref{eq:moebius} with three types
  of imaginary units appear from the action of the group \(\SL\) on
  the homogeneous space \(\SL/H\)~\cite{Kisil09c}, where \(H\) is any
  subgroup \(A\), \(N\), \(K\) from the Iwasawa
  decomposition~\eqref{eq:iwasawa-decomp}. Which other actions and
  hypercomplex numbers can be obtained from semisimple Lie groups and
  their subgroups?
\item Lobachevsky geometry of the upper half-plane is extremely beautiful and
  well-developed subject~\citelist{\cite{Beardon05a}
    \cite{CoxeterGreitzer}}. However the traditional study is limited
  to one subtype out of nine possible: with the complex numbers for
  M\"obius transformation and the complex imaginary unit used in
  FSCc~\eqref{eq:FSCc-matrix}. 
  The remaining eight cases shall be
  explored in various directions, notably in the context of discrete
  subgroups~\cite{Beardon95}. 
\item The Filmore-Springer-Cnops construction, see
  subsection~\ref{sec:invariance-fscc}, is closely related to the
  orbit method~\cite{Kirillov99} applied to \(\SL\). An extension of
  the orbit method from the Lie algebra dual to matrices representing
  cycles may be fruitful for semisimple Lie groups.
\end{enumerate}

\subsection{Analytic Functions}
\label{sec:analytic-functions-problems}
It is known that in several dimensions there are different notions of
analyticity, e.g. several complex variables and Clifford analysis.
However, analytic functions of a complex variable are
usually thought to be the only options in a plane domain. The
following seems to be promising:
\begin{enumerate}
\item \label{it:hyp-functions}
  Development of the basic components of analytic function theory
  (the Cauchy integral, the Taylor 
  expansion, the Cauchy-Riemann and Laplace equations, etc.) from the same
  construction and principles in the elliptic, parabolic and hyperbolic
  cases and subcases.
\item \label{it:bergman}
  Identification of Hilbert spaces of analytic functions of Hardy and
  Bergman types, investigation of their properties. Consideration of the
  corresponding T\"oplitz operators and algebras generated by them.
\item Application of analytic methods to elliptic, parabolic and hyperbolic
  equations and corresponding boundary and initial values problems. 
\item \label{it:mult-dim-funct}
  Generalisation of the results obtained to higher dimensional
  spaces. Detailed investigation of physically significant cases of three
  and four dimensions.
\end{enumerate}

\subsection{Functional Calculus}
\label{sec:functional-calculus-problems}
The functional calculus of a finite dimensional operator considered in
Section~\ref{sec:functional-calculus} is elementary but provides a
coherent and comprehensive treatment. It shall be extended to further
cases where other approaches seems to be rather limited.
\begin{enumerate}
\item Nilpotent and quasinilpotent operators have the most trivial
  spectrum possible (the single point \(\{0\}\)) while their structure
  can be highly non-trivial. Thus the standard spectrum is
  insufficient for this class of operators. In contract, the covariant
  calculus and the spectrum give complete description of nilpotent
  operators---the basic prototypes of quasinilpotent ones.  For
  quasinilpotent operators the construction will be more complicated
  and shall use analytic functions mentioned in \ref{it:hyp-functions}.
  
\item The version of covariant calculus described above is based on the
  \emph{discrete series} representations of \(\SL\) group and is particularly
  suitable for the description of the \emph{discrete spectrum} (note the
  remarkable coincidence in the names). 
  
  It is interesting to develop similar covariant calculi based on the
  two other representation series of \(\SL\): \emph{principal} and
  \emph{complementary}~\cite{Lang85}. The corresponding versions of
  analytic function theories for principal~\cite{Kisil97c} and
  complementary series~\cite{Kisil05a} were initiated within a
  unifying framework. The classification of analytic function theories
  into elliptic, parabolic, hyperbolic~\cite{Kisil05a,Kisil06a} hints
  the following associative chains:
  \begin{center}
    \begin{tabular}{c@{---}c@{---}c}
      \textbf{Representations of \(\SL\) } &  \textbf{ Function Theory } & 
      \textbf{ Type of Spectrum }\\[1mm] \hline\hline
      discrete series &  elliptic   & discrete spectrum\\
      principal series& hyperbolic & continuous spectrum\\
      complementary series & parabolic & residual spectrum
    \end{tabular}
  \end{center}

\item Let \(a\) be an operator with \(\spec a\in\bar{\Space{D}{}}\)
  and \(\norm{a^k}< C k^p\). It is typical to consider instead of
  \(a\) the \emph{power bounded} operator \(ra\), where \(0<r< 1\),
  and consequently develop its \(\FSpace{H}{\infty}\) calculus.
  However such a regularisation is very rough and hides the nature of
  extreme points of \(\spec{a}\). To restore full information a
  subsequent limit transition \(r\rightarrow 1\) of the regularisation
  parameter \(r\) is required. This make the entire technique rather
  cumbersome and many results have an indirect nature.

  The regularisation \(a^k\rightarrow a^k/k^p\) is more natural and
  accurate for polynomially bounded operators. However it cannot be
  achieved within the homomorphic calculus Defn.~\ref{de:calculus-old}
  because it is not compatible with any algebra homomorphism. Albeit
  this may be achieved within the covariant
  calculus~Defn.~\ref{de:functional-calculus-new} and Bergman type space
  from~\ref{it:bergman}.

\item Several non-commuting operators are especially difficult to
  treat with functional calculus Defn.~\ref{de:calculus-old} or a joint
  spectrum. For example, deep insights on joint spectrum of commuting
  tuples~\cite{JTaylor72} refused to be generalised to non-commuting
  case so far.  The covariant calculus was initiated~\cite{Kisil95i}
  as a new approach to this hard problem and was later found useful
  elsewhere as well.  Multidimensional covariant
  calculus~\cite{Kisil04d} shall use analytic functions described
  in~\ref{it:mult-dim-funct}.

\end{enumerate}

\subsection{Quantum Mechanics}
\label{sec:quantum-mechanics}

Due to the space restrictions we did not mentioned connections with
quantum mechanics~\citelist{\cite{Kisil96a} \cite{Kisil02e}
  \cite{Kisil05c} \cite{Kisil04a} \cite{Kisil09a} \cite{Kisil10a}}. In general
Erlangen approach is much more popular among physicists rather than
mathematicians. Nevertheless its potential is not exhausted even there. 

\begin{enumerate}
\item There is a possibility to build representation of the Heisenberg
  group using characters of its centre with values in dual and double
  numbers rather than in complex ones. This will naturally unifies
  classical mechanics, traditional QM and hyperbolic
  QM~\cite{Khrennikov08a}. 
\item  Representations of nilpotent Lie groups with multidimensional
  centres in Clifford algebras as a framework for consistent quantum
  filed theories based on De Donder--Weyl formalism~\cite{Kisil04a}.
\end{enumerate}
\begin{rem}
  This work is performed within the ``Erlangen programme at large''
  framework~\cites{Kisil06a,Kisil05a}, thus it would be suitable to
  explain the numbering of various papers. Since the logical order may be
  different from chronological one the following numbering  scheme
  is used:
  \begin{center}
  \begin{tabular}{||c|p{.7\textwidth}||}
    \hline\hline
    Prefix&Branch description\\
    \hline\hline
    ``0'' or no prefix & Mainly geometrical works, within the classical
    field of Erlangen programme by F.~Klein, see~\citelist{
    \cite{Kisil05a} \cite{Kisil09c}}\\
    \hline 
    ``1'' & Papers on analytical functions theories and wavelets, e.g.~\cite{Kisil97c}\\
    \hline
    ``2'' & Papers on operator theory, functional calculi and
    spectra, e.g.~\cite{Kisil02a}\\ 
    \hline 
    ``3'' & Papers on mathematical physics, e.g.~\cite{Kisil10a}\\
    \hline\hline
  \end{tabular}    
  \end{center}
  For example, this is the first paper in the mathematical physics area.
\end{rem}

\small

\providecommand{\CPP}{\texttt{C++}} \providecommand{\NoWEB}{\texttt{noweb}}
  \providecommand{\MetaPost}{\texttt{Meta}\-\texttt{Post}}
  \providecommand{\GiNaC}{\textsf{GiNaC}}
  \providecommand{\pyGiNaC}{\textsf{pyGiNaC}}
  \providecommand{\Asymptote}{\texttt{Asymptote}} \newcommand{\noopsort}[1]{}
  \newcommand{\printfirst}[2]{#1} \newcommand{\singleletter}[1]{#1}
  \newcommand{\switchargs}[2]{#2#1} \newcommand{\irm}{\textup{I}}
  \newcommand{\iirm}{\textup{II}} \newcommand{\vrm}{\textup{V}}
  \providecommand{\cprime}{'} \providecommand{\eprint}[2]{\texttt{#2}}
  \providecommand{\myeprint}[2]{\texttt{#2}}
  \providecommand{\arXiv}[1]{\myeprint{http://arXiv.org/abs/#1}{arXiv:#1}}
  \providecommand{\doi}[1]{\href{http://dx.doi.org/#1}{doi: #1}}

\end{document}